\documentclass[reqno,b5paper]{amsart}
\usepackage{amsmath}
\usepackage{amssymb}
\usepackage{amsthm}
\usepackage{enumerate}
\usepackage[mathscr]{eucal}
\theoremstyle{plain}
\newtheorem{thm}{Theorem}[section]

\newtheorem{Example}{Example}[section]
\newtheorem{note}{Note}[section]

\theoremstyle{definition}
\newtheorem{defn}{Definition}[section]

\newtheorem{op}{Open Problem}[section]

\begin{document}

\setcounter {page}{1}
\title{a note on rough statistical convergence of order $\alpha$}

\author[M. Maity]{Manojit Maity\ }
\newcommand{\acr}{\newline\indent}
\maketitle
\address{ 25 Teachers Housing Estate, P.O.- Panchasayar, Kolkata- 700094, West Bengal, India. Email: mepsilon@gmail.com\\}

\maketitle
\begin{abstract}
In this paper, in the line of Aytar\cite{Ay2} and \c{C}olak \cite{Co}, we introduce the notion of rough statistical convergence of order $\alpha$ in normed linear spaces and study some properties of the set of all rough statistical limit points of order $\alpha$.

\end{abstract}
\author{}
\maketitle
{ Key words and phrases :} Rough statistical convergence of order $\alpha$,
rough statistical limit points of order $\alpha$, \\

\textbf {AMS subject classification (2010) : 40A05, 40G99} .  \\

\section{\textbf{Introduction:}} The concept of statistical convergence was introduced by Steinhaus \cite{St} and Fast \cite{Fa} and later it was reintroduced by Schoenberg \cite{Sc} independently. Over the years a lot of works have been done in this area.
The concept of rough statistical convergence of single sequences was first introduced by S. Aytar \cite{Ay2}. Later the concept of statistical convergence of order $\alpha$ was introduced by R. $\c{C}$olak \cite{Co}.

If $ x = \{x_n\}_{n \in \mathbb{N}}$ is a sequence in some normed linear space $(X, \parallel . \parallel)$ and $r$ is a nonnegative real number then $x$ is said to rough statistical convergent to $\xi \in X$ if for any $\varepsilon > 0$, $ {\underset{n \rightarrow \infty}{\lim}} \frac{1}{n} | \{ k \leq n : \parallel x_k - \xi \parallel \geq r + \varepsilon \}| = 0 $, \cite{Ay2}.

For $ r = 0 $, rough statistical convergence coincides with statistical convergence.

In this paper following the line of Aytar \cite{Ay2} and $\c{C}$olak \cite{Co} we introduce the notion of rough statistical convergence of order $\alpha$ in normed linear spaces and prove some properties of the set of all rough statistical limit points of order $\alpha$.

\section{\textbf{Basic Definitions and Notations}}
\begin{defn}
Let $K$ be a subset of the set of positive integers $\mathbb{N}$. Let $K_n = \{k \in K : k \leq n\}$. 
Then the natural density of $K$ is given by $\underset{n \rightarrow \infty}{\lim}\frac{|K_n|}{n}$, where 
$|K_n|$ denotes the number of elements in $K_n$.
\end{defn}

\begin{defn}
Let $K$ be a subset of the set of positive integers $\mathbb{N}$ and $\alpha$ be any real number 
with $0< \alpha \leq 1$. Let $K_n = \{k \in K : k \leq n\}$.Then the natural density of order $\alpha$ of 
$K$ is given by $\underset{n \rightarrow \infty}{\lim}\frac{|K_n|}{n^{\alpha}}$, where 
$|K_n|$ denotes the number of elements in $K_n$.
\end{defn}

\begin{note}
Let $x = \{x_n\}_{n \in \mathbb{N}}$ be a sequence. Then $x$ satisfies some property P for all $k$ except a set 
whose natural density is zero. Then we say that the sequence $x$ satisfies P for almost all $k$ and we abbreviated this
by a.a.k.
\end{note}

\begin{note}
Let $x = \{x_n\}_{n \in \mathbb{N}}$ be a sequence. Then $x$ satisfies some property P for all $k$ except a set 
whose natural density of order $\alpha$ is zero. Then we say that the sequence $x$ satisfies P for almost all $k$ and we abbreviated this by a.a.k($\alpha$).
\end{note}

\begin{defn}
Let $x = \{x_n\}_{n \in \mathbb{N}}$ be a sequence in a normed linear space $(X,\parallel . \parallel)$ and $r$ be a nonnegative real number. Let $0 < \alpha \leq 1$ be given. Then $x$ is said to be rough statistical convergent of order $\alpha$ to $\xi \in X$, denoted by $ x \overset{st-r^{\alpha}}\longrightarrow \xi $ if for any $ \varepsilon > 0 $, $\underset{n \rightarrow \infty}{\lim}\frac{1}{n^{\alpha}} | \{ k \leq n : \parallel x_k - \xi \parallel \geq r+\varepsilon \} | = 0$, that is a.a.k.$(\alpha) \parallel x_k - \xi \parallel < r+\varepsilon$ for every $\varepsilon > 0$ and some $r > 0$. In this case $\xi$ is called a  $ r^{\alpha}$-st-limit of $x$. 

The set of all rough statistical convergent sequences of order $\alpha$ will be denoted by $rS^{\alpha}$ for fixed $r$ with $ 0 < r \leq 1$.
\end{defn}

Throughout this paper $X$ will denote a normed linear space and $r$ will denote a nonnegative real number and $x$ will denote the sequence $ x = \{x_n\}_{n \in \mathbb{N}}$ in $X$.

In general, the $ r^{\alpha} $-st-limit point of a sequence may not be unique. So we consider $r^{\alpha}$-st-limit set of a sequence $x$, which is defined by $st\mbox{-}LIM_x^{r^{\alpha}} = \{\xi \in X : x_n\overset{st-r^{\alpha}}\longrightarrow \xi \}$. The sequence $x$ is said to be $r^{\alpha}$-statistical convergent provided that $st\mbox{-}LIM_x^{r^{\alpha}} \neq \emptyset$.\\
For unbounded sequence rough limit set $LIM_x^r = \emptyset$.\\
But in case of rough statistical convergence of order $\alpha$ $st\mbox{-}LIM_x^{r^{\alpha}} \neq \emptyset$ even though the sequence is unbounded. For this we consider the following example.

\begin{Example}
Let $ X = \mathbb{R} $. We define a sequence in the following way,
\begin{eqnarray*}
x_n &=& (-1)^n~~:i~~ \neq~~ n^2,~~ \alpha=1 \\
    &=& n,~~ \mbox{otherwise}
\end{eqnarray*}\\
Then 
\begin{eqnarray*}
st\mbox{-}LIM_x^{r^{\alpha}} &=& \emptyset, ~~\mbox{if}~~ r<1 \\
                      &=& [1-r,r-1],~~\mbox{otherwise}.
\end{eqnarray*}
and $LIM_x^{r^{\alpha}}=\emptyset$ for all $r \geq 0$.
\end{Example}

\begin{defn}
A sequence $\{x_n\}_{n \in \mathbb{N}}$ is said to be statistically bounded if there exists a positive
real number $M$ such that 
$\underset{n \rightarrow \infty}{\lim}\frac{1}{n}|\{k \leq n : \parallel x_k \parallel \geq M \}| =0 $
\end{defn}

\begin{defn}
Let $0 < \alpha \leq 1 $ be given. A sequence $\{x_n\}_{n \in \mathbb{N}}$ is said to be 
statistically bounded of order $\alpha$ if there exists a positive real number $M$ such that 
$\underset{n \rightarrow \infty}{\lim}\frac{1}{n^{\alpha}}|\{k \leq n : \parallel x_k \parallel \geq M \}| =0 $
\end{defn}

\section{\textbf{Main Results}}

\begin{thm}
Let $x$ be a sequence in $X$. Then $x$ is statistically bounded of order $\alpha$ if and only if there exists a nonnegative real number $r$ such that $st\mbox{-}LIM_x^{r^{\alpha}} \neq \emptyset$.
\end{thm}

\begin{proof}
The condition is necessary.\\

Since the sequence $x$ is statistically bounded, there exists a positive real number $M$ such that $\underset{n \rightarrow \infty}{\lim}\frac{1}{n^{\alpha}} | \{k \leq n : \parallel x_k \parallel \geq M\}| = 0$.
Let $K = \{k \in \mathbb{N} : \parallel x_k \parallel \geq M \}$. Define $r^{'} = sup\{\parallel x_k \parallel : k \in K^c\}$. Then the set $st\mbox{-}LIM_x^{r^{'\alpha}}$ contains the origin of $X$. Hence $st\mbox{-}LIM_x^{r^{'\alpha}} \neq \emptyset$.\\

The condition is sufficient.\\

Let $st\mbox{-}LIM_x^{r^{\alpha}} \neq \emptyset$ for some $r \geq 0$. Then there exists $ l \in X $ such that $l \in st\mbox{-}LIM_x^{r^{\alpha}}$. Then $\underset{n \rightarrow \infty}{\lim}\frac{1}{n^{\alpha}} |\{k \leq n : \parallel x_k - l \parallel \geq r + \varepsilon\} | = 0$ for each $\varepsilon > 0$. Then we say that almost all $x_k$'s are contained in some ball with any radius greater than r. So the sequence $x$ is statistically bounded.
\end{proof}

\begin{thm}
If $x^{'}=\{x_{n_k}\}_{{k} \in \mathbb{N}}$ is a subsequence of $x=\{x_n\}_{n \in \mathbb{N}}$ then $st\mbox{-}LIM_x^{r^{\alpha}} \subseteq st\mbox{-}LIM_{x^{'}}^{r^{\alpha}}$.
\end{thm}

\begin{proof}
The proof is straight forward. So we omit it.
\end{proof}

\begin{thm}
$st\mbox{-}LIM^{r^{\alpha}}_x$, the rough statistical limit set of order $\alpha$ of a sequence $x$ is closed.
\end{thm}

\begin{proof}
If $st\mbox{-}LIM_x^{r^{\alpha}} = \emptyset$ then there is nothing to prove. So we assume that $st\mbox{-}LIM_x^{r^{\alpha}} \neq \emptyset$. We can choose a sequence $\{y_k\}_{k \in \mathbb{N}} \subseteq st\mbox{-}LIM_x^{r^{\alpha}}$ such that $y_k \longrightarrow y_*$ for $ k \rightarrow \infty $. It suffices to prove that $y_* \in st\mbox{-}LIM_x^{r^{\alpha}}.$\\
Let $\varepsilon > 0$. Since $y_k \longrightarrow y_*$ there exists $k_{\varepsilon} \in \mathbb{N}$ such that $\parallel y_k - y_* \parallel < \frac{\varepsilon}{2}$ for $k > k_{\varepsilon}$.
Now choose $k_0 \in \mathbb{N}$ such that $k_0 > k_{\varepsilon}$. Then we can write $\parallel y_{k_0} - y_* \parallel < \frac{\varepsilon}{2}$.
Again since $\{y_k\}_{k \in \mathbb{N}} \subseteq st\mbox{-}LIM_x^{r^{\alpha}}$ we have $y_{k_0} \in st\mbox{-}LIM_x^{r^{\alpha}}$. This implies
\begin{equation}
\underset{n \rightarrow \infty}{\lim}\frac{1}{n^{\alpha}} | \{k \leq n : \parallel x_k - y_{k_0} \parallel \geq r + \frac{\varepsilon}{2} \} | = 0. 
\end{equation}
Now we show the inclusion
\begin{equation}
\{k \leq n: \parallel x_k - y_{k_0} \parallel < r + \frac{\varepsilon}{2} \} \subseteq \{k \leq n: \parallel x_k - y_* \parallel < r + \varepsilon \}
\end{equation}
holds.\\
Choose $j \in \{k \leq n: \parallel x_k - y_{k_0} \parallel < r + \frac{\varepsilon}{2}\}$. Then we have $\parallel x_j - y_{k_0} \parallel < r + \frac{\varepsilon}{2}$ and hence $\parallel x_j - y_* \parallel \leq \parallel x_j - y_{k_0} \parallel + \parallel y_{k_0} - y_* \parallel < r + \varepsilon$ which implies $j \in \{k \leq n: \parallel x_k -y_* \parallel < r + \varepsilon\}$, which proves the inclusion (2). \\
From (1) we can say that the set on the right hand side of(2) has natural density 1. Then the set on the left hand side of (2) must have natural density 1. Hence we get $\underset{n \rightarrow \infty}{\lim} \frac{1}{n^{\alpha}} |\{k \leq n: \parallel x_k - y_* \parallel \geq r + \varepsilon \} | = 0$.\\
This completes the proof.
\end{proof}

\begin{thm}
Let $x$ be a sequence in $X$. Then the rough statistical limit set of order $\alpha$  $st\mbox{-}LIM^{r^{\alpha}}_x$ is convex.
\end{thm}

\begin{proof}
Choose $y_1, y_2 \in st\mbox{-}LIM_x^{r^{\alpha}}$ and let $\varepsilon > 0$. Define $K_1 = \{k \leq n : \parallel x_k - y_1 \parallel \geq r + \varepsilon\}$ and $K_2 = \{k \leq n: \parallel x_k - y_2 \parallel \geq r + \varepsilon\}$. Since $y_1, y_2 \in st\mbox{-}LIM_x^{r^{\alpha}}$, we have $\underset{n \rightarrow \infty}{\lim}\frac{1}{n^{\alpha}} | K_1 | = \underset{n \rightarrow \infty}{\lim}\frac{1}{n^{\alpha}} | K_2 | = 0$. Let $\lambda$ be any positive real number with $0 \leq \lambda \leq 1$.\\ 
Then 
\begin{eqnarray*}
\parallel x_k - [(1 - \lambda)y_1 + {\lambda}y_2]\parallel = \parallel (1 - \lambda)(x_k - y_1) + {\lambda}(x_k - y_2) \parallel < r + \varepsilon
\end{eqnarray*}
for each $k \in {K_1}^c \cap {K_2}^c$. Since $\underset{n \rightarrow \infty}{lim}\frac{1}{n^{\alpha}} | {K_1}^c \cap {K_2}^c | = 1$, we get
\begin{eqnarray*}
\underset{n \rightarrow \infty}{\lim} \frac{1}{n^{\alpha}}| \{k \leq n: \parallel x_k -[(1 - \lambda)y_1 + {\lambda}y_2] \parallel \geq r + \varepsilon \} | = 0
\end{eqnarray*}
that is
\begin{eqnarray*}
[(1 - \lambda)y_1 + (\lambda)y_2] \in st\mbox{-}LIM_x^{r^{\alpha}}
\end{eqnarray*}
which proves the convexity of the set $st\mbox{-}LIM_x^{r^{\alpha}}$.
\end{proof}

\begin{thm}
Let $x$ be a sequence in $X$ and $r > 0$. Then the sequence $x$ is rough statistical convergent of order $\alpha$ to $\xi \in X $ if and only if there exists a sequence $y = \{y_n\}_{n \in \mathbb{N}}$ in $X$ such that $y$ is statistically convergent of order $\alpha$ to $\xi$ and $\parallel x_n - y_n \parallel \leq r$ for all $n \in \mathbb{N}$.
\end{thm}

\begin{proof}
The condition is necessary.\\

Let $x_n \overset{st\mbox{-}r^{\alpha}}\longrightarrow \xi$. Choose any $\varepsilon > 0$. Then we have
\begin{eqnarray}
\underset{n \rightarrow \infty}{\lim}\frac{1}{n^{\alpha}}|\{k \leq n: \parallel x_k - \xi \parallel \geq r + \varepsilon\}| = 0  ~~\mbox{for some} ~~r > 0.
\end{eqnarray}\\
Now we define \\
\begin{eqnarray*}
y_n &=& \xi,~~\mbox{if}~~ \parallel x_n - \xi \parallel \leq r \\
    &=& x_n + r \frac{\xi - x_n}{\parallel x_n - \xi \parallel},~~\mbox{otherwise}
\end{eqnarray*}\\
Then we can write \\
\begin{eqnarray*}
\parallel y_n - \xi \parallel &=& 0,~~\mbox{if}~~ \parallel x_n - \xi \parallel \leq r \\
                             &=& \parallel x_n - \xi \parallel - r,~~\mbox{otherwise}
\end{eqnarray*}\\
and by definition of $y_n$, we have $\parallel x_n - y_n \parallel \leq r$ for all $n \in \mathbb{N}$.
Hence by (3) and the definition of $y_n$ we get $\underset{n \rightarrow \infty}{\lim} \frac{1}{n^{\alpha}}|\{ k \leq n: \parallel y_k - \xi \parallel \geq \varepsilon \}| = 0$. Which implies that the sequence $\{y_n\}_{n \in \mathbb{N}}$ is statistically covergent of order $\alpha$ to $\xi$.\\

The condition is sufficient.\\

Since $\{y_n\}_{n \in \mathbb{N}}$ is statistically convergent of order $\alpha$ to $\xi$, we have \\ 
$\underset{n \rightarrow \infty}{\lim}\frac{1}{n^{\alpha}} |\{k \leq n: \parallel y_k - \xi \parallel \geq \varepsilon\}| = 0$ for all $\varepsilon > 0$. Also since for a given $r > 0$ and for the sequence $x = \{x_n\}_{n \in \mathbb{N}}$ $\parallel x_n -y_n \parallel < r$, the inclusion $\{k \leq n: \parallel x_k - \xi \parallel \geq r + \varepsilon\} \subseteq \{k \leq n: \parallel y_k - \xi \parallel \geq \varepsilon\}$ holds. Hence we get $\underset{n \rightarrow \infty}{\lim}\frac{1}{n^{\alpha}} |\{k \leq n: \parallel x_k - \xi \parallel \geq r + \varepsilon\}| = 0$.\\
This completes the proof.
\end{proof}

\begin{thm}
For an arbitrary $c \in {\Gamma}_x$, where ${\Gamma}_x$ is the set of all rough statistical cluster points of a sequence  $x \in X$, we have for a positive real number $r$, $\parallel \xi -c \parallel \leq r$ for all $\xi \in st\mbox{-}LIM_x^{r^{\alpha}}$.
\end{thm}

\begin{proof}
Let $ 0 < \alpha \leq 1 $ be given. On the contrary let assume that there exists a point $ c \in {\Gamma}_x $ and $ \xi \in st\mbox{-}LIM^{r^{\alpha}}_x $ such that $ \parallel \xi - c \parallel > r $. Choose $ \varepsilon = \frac{\parallel \xi - c \parallel - r}{3}$. Then
\begin{eqnarray}
\{k \leq n : \parallel x_k - \xi \parallel \geq r + \varepsilon \} \supseteq \{k \leq n : \parallel x_k - c \parallel < \varepsilon \}
\end{eqnarray}
holds. Since $ c \in {\Gamma}_x $, we have $\underset {n \rightarrow \infty}{\lim} \frac{1}{n^{\alpha}} |\{k \leq n : \parallel x_k - c \parallel < \varepsilon\}| \neq 0 $. Hence by (4) we have $\underset {n \rightarrow \infty}{\lim}\frac{1}{n^{\alpha}} |\{k \leq n : \parallel x_k - \xi \parallel < r + \varepsilon\}| \neq 0 $. This is a contradictin to the fact $ \xi \in st\mbox{-}LIM^{r^{\alpha}}_x $.
\end{proof}

\begin {thm}
Let $x$ be sequence in the strictly convex space. Let $r$ and $\alpha$ be two positive real numbers. If for any $y_1, y_2 \in st\mbox{-}LIM_x^{r^{\alpha}}$ with $\parallel y_1 - y_2 \parallel = 2r$, then $x$ is statistically convergent of order $\alpha$ to $\frac{y_1 + y_2}{2}$.
\end{thm}

\begin{proof}
Let $z \in {\Gamma}_x$. Then for any $y_1,y_2 \in st-LIM_x^{r^{\alpha}}$ implies
\begin{equation}
\parallel y_1 - z \parallel \leq r ~~\mbox{and}~~ \parallel y_2 - z \parallel \leq r.
\end{equation}\\
On the other hand we have
\begin{equation}
2r = \parallel y_1 - y_2 \parallel \leq \parallel y_1 - z \parallel + \parallel y_2 - z \parallel.
\end{equation}\\
Hence by (5) and (6) we get $\parallel y_1 - z \parallel = \parallel y_2 - z \parallel = r$. Since 
\begin{equation}
{\frac{1}{2}}(y_2 - y_1) = {\frac{1}{2}}[(z - y_1)+(y_2 -z)]
\end{equation}\\
and $\parallel y_1 - y_2 \parallel = 2r$, we get $\parallel {\frac{1}{2}}(y_2 -y_1) \parallel = r$.
By strict convexity of the space and from the equality(7) we get ${\frac{1}{2}}(y_1 - y_2) = z - y_1 = y_2 - z $, which implies that $z = {\frac{1}{2}}(y_1 + y_2)$. Hence $z$ is unique statistical cluster point of the sequence $x$. On the other hand, from the assumption $y_1, y_2 \in st\mbox{-}LIM_x^{r^{\alpha}}$ implies that $st\mbox{-}LIM_x^{r^{\alpha}} \neq \emptyset$. So by Theorem 3.1 the sequence $x$ is statistically bounded of order ${\alpha}$. Since $z$ is the unique statistical cluster point ot the statistically bounded sequence $x$ of order $\alpha$ we have the sequence $x$ is statistically convergent to $z = {\frac{1}{2}}(y_1 +y_2)$. 
\end{proof}

\begin{thm}
Let $0 < \alpha \leq 1$ and $x$ and $y$ be two sequences. Then\\
(i) if $r^{\alpha}\mbox{-}st\mbox{-}\lim x = x_0$ and $c \in \mathbb{R}$, then $r^{\alpha}\mbox{-}st\mbox{-}\lim cx = cx_0$\\
(ii)if $r^{\alpha}\mbox{-}st\mbox{-}\lim x = x_0$ and $r^{\alpha}\mbox{-}st\mbox{-}\lim y = y_0$ then $r^{\alpha}\mbox{-}st\mbox{-}\lim(x + y) = x_0 +y_0$.  
\end{thm}

\begin{proof}
(i) If $c =0$ it is trivial. Suppose that $c \neq 0$. Then the proof of (i) follows from
$\frac{1}{n^{\alpha}}|\{k \leq n : \parallel cx_k - cx_0 \parallel \geq r + \varepsilon \}|
\leq \frac{1}{n^{\alpha}}|\{k \leq n : \parallel x_k - x_0 \parallel \geq \frac{r + \varepsilon}{|c|} \}| = \frac{1}{n^{\alpha}} |\{k \leq n : \parallel x_k - x_0 \parallel \geq \frac{r}{|c|} + \frac{\varepsilon}{|c|} \}|$. Since $x$ is rough statistical convergent of order $\alpha$, hence $cx$ is also rough statistical convergent of order $\alpha$.\\
Again
\begin{eqnarray*}
\frac{1}{n^{\alpha}}|\{ k \leq n : \parallel (x_k + y_k)-(x_0 + y_0)\parallel \geq r + \varepsilon\}|
 \leq \frac{1}{n^{\alpha}}|\{ k \leq n : \parallel x_k - x_0 \parallel \geq r + \varepsilon\}|\\
+ \frac{1}{n^{\alpha}}|\{ k \leq n : \parallel y_k - y_0 \parallel \geq r + \varepsilon\}|
\end{eqnarray*}
\end{proof}

It is easy to see that every convergent sequence is rough statistical convergent of order $\alpha$, but the converse is not true always.

\begin{Example}
Let us consider the following sequence of real numbers defined by,
\begin{eqnarray*}
x_k &=& 1, ~~\mbox{if}~~k=n^3 \\
    &=& 0, ~~\mbox{otherwise}
\end{eqnarray*}
Then it is easy to see that the sequence is rough statistical convergent of order $\alpha$ with $rS^{\alpha}\mbox{-}\lim x_k = 0$ for $\alpha > \frac{1}{3}$, but it is not a convergent sequence.    
\end{Example}

\begin{thm}
Let $0 < \alpha \leq \beta \leq 1$. Then $rS^{\alpha} \subseteq rS^{\beta}$, where $rS^{\alpha}$ and $rS^{\beta}$ denote the set of all rough statistical convergent sequence of order $\alpha $ and $\beta $ respectively. 
\end{thm}

\begin{proof}
If $0 < \alpha \leq \beta \leq 1$ then 
\begin{eqnarray*}
\frac{1}{n^{\beta}} |\{ k \leq n : \parallel x_k - l\parallel \geq r + \varepsilon \}| \leq \frac{1}{n^{\alpha}} |\{ k \leq n : \parallel x_k - l\parallel \geq r + \varepsilon \}|
\end{eqnarray*}
for every $\varepsilon > 0 $ and some $ r > 0 $ with limit $ l $. \\
Clearly this shows that $rS^{\alpha} \subseteq rS^{\beta}$.
\end{proof}

We do not know whether for a sequence $ x $ in $ X $, $ r > 0 $ and for  $ 0 < \alpha < 1 $, $ diam(st\mbox{-}LIM_x^{r^{\alpha}}) \leq 2r $ is true or not. 
  
So we leave this above fact as an open problem.


\begin{op}
Is it true for a sequence $ x $ in $ X $, $ r > 0 $ and $ 0 < \alpha < 1 $, $ diam(st\mbox{-}LIM_x^{r^{\alpha}}) \leq 2r $. 
\end{op}
Acknowledgement: I express my gratitude to Prof. Salih Aytar, Suleyman Demirel University, Turkey, for his paper entitled ``Rough Statistical Convergence'' which inspired me to prepare and develope this paper. I also express my gratitude to Prof. Pratulananda Das, Jadavpur University, India and Prof. Prasanta Malik, Burdwan University, India for their advice in preparation of this paper.

\end{document}